\documentclass[12pt,leqno]{amsart}
\usepackage[latin1]{inputenc}
\usepackage{color}
\usepackage{amssymb,amsmath,amsthm,amscd,amsfonts}
\usepackage{enumerate}
\usepackage[all]{xy}
\usepackage{hyperref,cleveref,graphics,mathrsfs}
\newtheorem{thm}{Theorem}[section]

\newtheorem{lem}[thm]{Lemma}

\theoremstyle{definition}

\newtheorem{rem}[thm]{Remark}

\numberwithin{equation}{section}

\def\epsilon{\varepsilon}

\usepackage{bbm}

\begin{document}
\allowdisplaybreaks
\title{No pure capillary solitary waves exist in 2D finite depth }

\author{ Mihaela Ifrim}
\address{Department of Mathematics\\
University of Wisconsin - Madison
} \email{ifrim@wisc.edu}

\author{Ben Pineau}
\address{Department of Mathematics\\
University of California at Berkeley
} \email{bpineau@berkeley.edu}
\author{Daniel Tataru}
\address{Department of Mathematics\\
University of California at Berkeley
} \email{tataru@math.berkeley.edu}
\author{Mitchell A.\ Taylor}
\address{Department of Mathematics\\
University of California at Berkeley
} \email{mitchelltaylor@berkeley.edu}

 \begin{abstract}
We prove that the 2D finite depth capillary water wave equations admit no solitary wave solutions. This closes the existence/non-existence problem for solitary water waves in 2D, under the classical assumptions of incompressibility and irrotationality, and with the physical parameters being gravity, surface tension and the fluid depth. 
\end{abstract} 

 
 \maketitle
 
 \bigskip

 \section{Introduction}

Solitary water waves are localized disturbances of a fluid surface which travel at constant speed and with a fixed profile. Such waves were  first observed by Russell in the mid-19th century \cite{Rus44}, and are fundamental features of many water wave models. The objective of this paper is to  settle the existence/non-existence problem for the full irrotational water wave system in 2D, with the physical parameters being gravity, surface tension, and the fluid depth. Five of the six combinations have already been dealt with, and the results are summarized in Table 1 - it is our intent to fill in the missing case.
\\

\begin{center}
\begin{tabular}{ |p{3cm}|p{3cm}|p{3cm}|p{3cm}|p{3cm}|  }

 \multicolumn{4}{c}{Table 1: Existence of 2D solitary waves in irrotational fluids} \\
 \hline
Gravity   & Capillarity & Depth & Existence\\
 \hline
 Yes   & Yes    & Infinite & Yes\\
  Yes   & No    & Infinite & No\\
   No   & Yes    & Infinite & No\\
  Yes   & Yes    & Finite & Yes\\
   Yes   & No    & Finite & Yes\\
    No   & Yes    & Finite & Unknown\\

 \hline
\end{tabular}
\end{center}
\medskip
In a nutshell, our result can be loosely formulated as follows:
\begin{thm}\label{informal}
No solitary waves exist in finite depth for the pure capillary irrotational water wave problem in 2D, even without the assumption that the free surface is a graph.
\end{thm}
A more precise formulation of the result is given later, in Theorem~\ref{Main}.

\subsection*{Historical perspectives} The mathematical study of travelling waves has been a fundamental - and longstanding - problem in fluid dynamics. Perhaps the first rigorous construction of 2D finite depth pure gravity solitary waves occurred in \cite{FriedrichsHyers54, Lavrentiev43}; further refinements can be found in \cite{Beale77, Mielke88}. Solitary waves with large  amplitudes were first constructed by Amick and Toland~\cite{AmickToland81b} in 1981 using global bifurcation techniques, leading to the existence of a limiting extreme wave with an angled crest~\cite{AmickFraenkelToland82}; see also~\cite{AmickToland81a,BenjaminBonaBose90, MR513927}. By now, a vast literature exists on this subject, including both results for gravity and for gravity-capillary waves (\cite{MR963906,MR1720395,MR1378603,MR2867413,
MR1949969,MR2379653, MR1133302, MR2263898}). For water waves in deep water, solitary waves
have been proved to exist provided that both gravity and
surface tension are present, see \cite{MR2073504,MR2069635,MR2847283, MR1423002}, 
following numerical work in 
\cite{MR988299,MR990170}. 
\\

The non-existence of 2D pure gravity solitary waves in infinite depth was originally proved in \cite{Hur},  under certain decay assumptions. The proof uses conformal mapping techniques, and the decay assumptions ultimately stem from difficulties in estimating commutators involving the Hilbert transform. The decay assumptions were completely removed in \cite{IT19}, as the authors were able to effectively deal with the aforementioned commutator issues - see \cite[Lemma 3.1]{IT19}. 
\\

The proof of our result is loosely based on the ideas of \cite{IT19}. The key difference is that the Tilbert transform (see \Cref{Sec} for the definition) does not enjoy the same commutator structure as the Hilbert transform. More precisely, we cannot simply replace Hilbert transforms with Tilbert transforms in \cite[Lemma 3.1]{IT19}. To circumvent this, we morally view the Tilbert transform as the Hilbert transform at high frequency, and a derivative at low frequency, and use these distinct regimes to close our argument. 
\\

For context, we mention that the problem we are considering in this article goes at least as far back as \cite{GS72}. More specifically, in \cite{GS72} it is noted that the systematic existence methods developed in \cite{F48, FriedrichsHyers54, K48} for the pure gravity problem in shallow water are unable to produce pure capillary solitary waves, but can be modified to produce gravity-capillary solitary waves. One may contrast the question of existence of solitary waves
with that of the existence of periodic travelling waves. 
Indeed, for pure capillary irrotational waves
in both finite and infinite depth, periodic travelling waves
are known to exist. These are called Crapper waves, 
and are quite explicit; see \cite{MR91075,MR475234}
for the original results, and also the survey in \cite{MR1869386}. Interestingly, the free surfaces of the Crapper waves need not be graphs, which makes the lack of graph assumption in \Cref{informal} essential. The reader is referred to \cite{MR3145074, Crapper-pert,  MR2999400, MR3020910, MR2292728} for further literature on pure capillary waves, as well as gravity-capillary perturbations of the Crapper waves.
\\

Finally, we mention two recent directions that are somewhat outside the scope of this paper. The first is the study of steady water waves with vorticity, for which we refer the interested reader to the surveys~\cite{groves:survey, strauss:survey}. As mentioned, our non-existence proof utilizes holomorphic coordinates, a technique which is not compatible with variable vorticity. However, such a restriction is quite natural, as heuristics dictate that one should expect solitary waves in problems with, say, constant non-zero vorticity.  The other interesting direction - in situations where solitary waves are known to exist - is to determine which speeds are capable of sustaining solitary waves. Recently, it was shown in \cite{KLW20} that all finite depth, irrotational, pure gravity solitary waves must obey the inequality $c^2>gh.$ Here $c$ is the speed, $g$ the gravitational constant, and $h$ the asymptotic depth. Heuristically, this result says that speeds that are precluded by the linearized problem are also precluded in the nonlinear problem.

\subsection*{Acknowledgements}Mihaela Ifrim was supported by a Luce Assistant Professorship, by the Sloan Foundation, and by an NSF CAREER grant DMS-1845037. Daniel Tataru was supported by the NSF grant DMS-1800294 as well as by a Simons Investigator grant from the Simons Foundation.  This material is also based upon work supported by the National Science Foundation under Grant No. DMS-1928930 while all four authors participated in the program \textit{Mathematical problems in fluid dynamics} hosted by the Mathematical Sciences Research Institute in Berkeley, California, during
the Spring 2021 semester.

\section{The equations in Eulerian coordinates}
 We consider the incompressible, finite depth water wave equations in two space dimensions. The motion of the water is governed by the incompressible Euler equations, with boundary conditions on the water surface and the  flat, finite bottom. We emphasize that this section is purely for motivational purposes, and is not the formulation we will use to prove our non-existence result. In particular, for simplicity, this subsection assumes that $\Gamma(t)$ is a graph, but we will \emph{not} assume this when working  with the holomorphic formulation of our problem.
\\

To describe the equations, denote the water domain at time $t$ by $\Omega(t)\subseteq \mathbb{R}^2$; we assume that $\Omega(t)$ has a flat finite bottom $\{y=-h\}$, and let $\eta(x,t)$ denote the height of the free surface as a function of the horizontal coordinate:
\begin{equation}\label{Domain}
    \Omega(t)=\{(x,y)\in \mathbb{R}^2 : -h<y<\eta(x,t)\}.
\end{equation}
The free surface of the water at time $t$ will be denoted by $\Gamma(t)$. As we are interested in solitary waves, we think of $\Gamma(t)$ as being asymptotically flat at infinity to $y\approx 0$. Since the 2D finite depth capillary water wave equations do permit periodic travelling waves, this decay at infinity will factor heavily into our proof, even though we do not impose any specific  rate of decay. 
\\

We denote by $u$ the fluid velocity and by $p$ the pressure. The vector field $u$ solves Euler's equations inside $\Omega(t),$

\begin{equation} \label{Euler} 
\left\{
\begin{aligned}
	& u_t+u\cdot \nabla u=-\nabla p-gj,\\
	& \text{div}\ u=0,\\
	& u(0,x)=u_0(x),
\end{aligned}\right.
\end{equation}
and the bottom boundary is impenetrable:
\begin{equation}
  u\cdot j=0 \ \ \text{when}\ y=-h.
    \end{equation}
On the upper boundary the atmospheric pressure is normalized to zero and we have the dynamic boundary condition
\begin{equation}
    p=-\sigma\textbf{H}(\eta) \ \ \ \text{on} \ \Gamma(t),
\end{equation}
and the kinematic boundary condition 
\begin{equation}
    \partial_t+u\cdot \nabla \ \ \text{is tangent to} \ \bigcup \Gamma(t).
\end{equation}
Here $g\geq 0$ represents the gravity, \begin{equation}\label{curvature} \textbf{H}(\eta)=\partial_x\left(\frac{\eta_x}{\sqrt{1+\eta_x^2}}\right)
\end{equation} is the mean curvature of the free boundary, and $\sigma> 0$ represents the surface tension coefficient. 
\\

We adhere to the classical assumption that the flow is irrotational, so we can write $u$ in terms of a velocity potential $\phi$ as $u=\nabla \phi$. It is easy to see that $\phi$ is a harmonic function whose normal derivative is zero on the bottom. Thus, $\phi$ is determined by its trace $\psi=\phi|_{\Gamma(t)}$ on the free boundary $\Gamma(t).$ Under these assumptions, it is well-known that the fluid dynamics can be expressed in terms of a one-dimensional evolution of the pair of variables $(\eta,\psi)$ via:

\begin{equation} \label{Zakharov} 
\left\{
\begin{aligned}
	& \partial_t\eta-G(\eta)\psi=0,\\
	& \partial_t\psi+g\eta-\sigma \textbf{H}(\eta)+\frac{1}{2}|\nabla\psi|^2-\frac{1}{2}\frac{(\nabla\eta\cdot \nabla\psi+G(\eta)\psi)^2}{1+|\nabla\eta|^2}=0.
\end{aligned}\right.
\end{equation}
Here $G$ denotes the Dirichlet to Neumann map associated to the fluid domain. This operator is one of the main analytical obstacles in this formulation of the problem, and in the next subsection we briefly discuss a change of coordinates that somewhat simplifies the analysis.
\\

We now write down the solitary wave equations. We begin with the equations \eqref{Domain}-\eqref{curvature} as well as the irrotationality condition, and assume that the profile is uniformly translating in the horizontal direction with velocity $c$, i.e., $\phi(x,y,t)=\phi_0(x-ct,y)$, $\eta(x,y,t)=\eta_0(x-ct,y),$ and $p(x,y,t)=p_0(x-ct,y).$ This gives the steady water wave equations. To get to solitary waves (as opposed to, say, periodic waves), we impose some averaged decay on $\eta_0$ and $u_0$, so that in the far-field the water levels out and is essentially still. Contrary to many works which use a frame of reference travelling with the localized disturbance, we choose a frame so that the fluid is at rest near infinity. This allows us to set to zero the integration constant in the Bernoulli equation; the price to pay is that there are terms with $c$ in the equations below.
\\

We are thus interested in states $(\eta,\phi)$ satisfying the following equations:

\begin{equation}\label{Laplace}
    \Delta\phi=0 \ \ \ \text{in} \ \Omega=\{(x,y)\in \mathbb{R}^2 : -h<y<\eta(x)\},
\end{equation}
\begin{equation}\label{Bernoulli}
    -c\phi_{x}+\frac{1}{2}|\nabla\phi|^2+g\eta-\sigma \partial_x\left(\frac{\eta_{x}}{\sqrt{1+\eta_{x}^2}}\right)=0\ \ \  \text{on} \  \Gamma=\{(x,y)\in \mathbb{R}^2 : y=\eta(x)\},
\end{equation}
\begin{equation}
    \phi_{y}=0\ \ \ \text{when} \ y=-h,
\end{equation}
\begin{equation}\label{Boundary phi}
    -c\eta_{x}+\phi_{x}\eta_{x}=\phi_{y}\ \ \  \text{on} \  \Gamma.
\end{equation}

We prove that the above equations admit no non-trivial solutions, with appropriate (averaged) decay at infinity. Such a claim, of course, presupposes certain regularity requirements on the solutions, but this will not play a major role due to ellipticity. Indeed, the above system 
can be shown to be locally elliptic whenever $(\eta,\phi)$ is above critical regularity, 
which corresponds to $\eta \in~H^{\frac32+}_{loc}$.
\\

\section{The equations in holomorphic coordinates}\label{Sec} As mentioned, one of the main difficulties of \eqref{Zakharov} is the presence of the Dirichlet to Neumann operator $G(\eta)$, which depends on the free boundary. For this reason, we will reformulate the equations in holomorphic coordinates, which, in some sense, diagonalizes $G(\eta)$. We will only highlight briefly the procedure of changing coordinates; full details can be found in \cite{BIT16}. Moreover, although \eqref{Zakharov} assumes that $\Gamma(t)$ is a graph, the formulation below does not require this, which is another advantage of this approach. 
\\

The conditions we require on $\Gamma(t)$ are the same (or weaker, see the discussion below) as those listed in Section 2.3 of \cite{BIT16}; namely, that $\Gamma(t)$ can be parametrized to have sufficient Sobolev regularity, has no degeneracies or self-intersections, and never touches the bottom boundary. These assumptions are used in \cite[Theorem 3]{BIT16} to justify the existence of the conformal map we refer to below.
\\

In the holomorphic setting, the coordinates are denoted by $\alpha+i\beta\in S:=\mathbb{R}\times (-h,0)$, and the fluid domain is parameterized by the conformal map
$$z:S\to \Omega(t),$$
which takes the bottom $\mathbb{R}-ih$ into the bottom, and the top $\mathbb{R}$ into the top $\Gamma(t)$. The restriction of this map to the real line is denoted by $Z$, i.e., $Z(\alpha):=z(\alpha-i0)$, and can be viewed as a parametrization of the free boundary $\Gamma(t)$. We will work with the variables $W(\alpha)=Z(\alpha)-\alpha$, and the trace $Q(\alpha)$ of the holomorphic velocity potential on the free surface. $W$ and $Q$ are traditionally called holomorphic functions, which in this terminology means that they can be realized as the trace on the upper boundary $\beta=0$ of holomorphic functions in the strip $S$ which are purely real on the lower boundary $\beta=-h$. The space of holomorphic functions is a real algebra, but is not a complex algebra.
\\

In terms of regularity, we note that the existence of the conformal map is guaranteed by the Riemann
Mapping Theorem  for any simply connected fluid domain. In order to have an equivalence between 
Sobolev norms, it suffices to assume that the free surface $\Gamma$ has critical Besov regularity $B^{\frac32}_{2,1}$. This, in particular, guarantees that $\Gamma$ is a graph outside of a compact set. The conformal map, then, has the matching property $\Im (W) \in B^{\frac32}_{2,1}$, and in particular 
$\Im (W)$ and  $W_\alpha$ are bounded. For more details we refer the reader to both  \cite[Section 2]{BIT16} and the stronger results in \cite{AIT}, as well as the more general local results of \cite{MR1986694}. 
\\

The two-dimensional finite depth gravity-capillary water wave equations in holomorphic coordinates can be written as follows:
\begin{equation}\label{holo}
\begin{cases}
 W_t+F(1+W_\alpha)=0,\\
 Q_t+FQ_\alpha-g\mathcal{T}_h[W]+\textbf{P}_h\left[\frac{|Q_\alpha|^2}{J}\right]+\sigma \textbf{P}_h\left[i\left(\frac{W_{\alpha\alpha}}{J^{1/2}(1+W_
 \alpha)}-\frac{\overline{W_{\alpha\alpha}}}{J^{1/2}(1+\overline{W_\alpha})}\right)\right]=0,
 \end{cases}
 \end{equation}

where \begin{equation}
    J=|1+W_\alpha|^2
\end{equation}
and \begin{equation}
    F=\textbf{P}_h\left[\frac{Q_\alpha-\overline{Q_\alpha}}{J}\right].
\end{equation}
As before, $g$ and $\sigma$ are non-negative parameters, at least one of which is non-zero. $\mathcal{T}_h$ denotes the Tilbert transform, which is the Fourier multiplier with symbol $-i\tanh(h\xi)$, and arises in order to characterize what it means to be a holomorphic function. Precisely, holomorphic functions are described by the relation
\begin{equation}
    \Im(u)=-\mathcal{T}_h \Re(u).
\end{equation}
It is important to note that the Tilbert transform takes real-valued functions to real-valued functions, and satisfies the following product rule:
\begin{equation}
    u\mathcal{T}_h[v]+\mathcal{T}_h[u]v=\mathcal{T}_h[uv-\mathcal{T}_h[u]\mathcal{T}_h[v]].
\end{equation}

Finally, $\textbf{P}_h$ is the projection onto the space of holomorphic functions. In terms of $\mathcal{T}_h$ it can be written as
\begin{equation}\label{proj form}
   \textbf{P}_hu=\frac{1}{2}\left[(1-i\mathcal{T}_h)\Re(u)+i(1+i\mathcal{T}_h^{-1})\Im(u)\right].
\end{equation}

In the case of no surface tension, equations \eqref{holo} were derived in \cite{BIT16}. We begin with a brief outline of how the surface tension term arises, as we are particularly interested in the case when $g=0$ and $\sigma>0$. 
\\

Following \cite{BIT16}, we arrive at the Bernoulli equation
\begin{equation}
    \phi_t+\frac{1}{2}|\nabla\phi|^2+gy+p=0.
\end{equation}
We then evaluate this equation on the top boundary and apply the dynamic boundary condition to replace $p$ by $-\sigma\textbf{H}.$ We then pass to the strip $S$ - so the equations are now defined on $\{\beta=0\}$ - rewrite the equations in terms of the holomorphic variables, clear common factors of $2$, and project. Running this procedure explicitly for the term with $\sigma$, we begin by parameterizing $\Gamma(t)$ by, say, $s\mapsto (\gamma_1(s),\gamma_2(s))$ and write $-\sigma \textbf{H}$ in the standard parametric way. We then use the relations

$$\gamma_1(s)=\Re(Z(\alpha)), \ \ \ \gamma_2(s)=\Im(Z(\alpha))$$
and formal calculations to write the capillary expression in terms of the holomorphic variables as: $$\sigma i\left(\frac{W_{\alpha\alpha}}{J^{1/2}(1+W_
 \alpha)}-\frac{\overline{W_{\alpha\alpha}}}{J^{1/2}(1+\overline{W_\alpha})}\right),$$
 which after projecting gives us the capillary term in \eqref{holo}.

\begin{rem}\label{Ambig}
Before proceeding, we would like to point out some inherent ambiguities of the above equations, which have to be properly interpreted. The first stems from the horizontal translation symmetry of the strip, which causes some arbitrariness in the choice of conformal mapping; precisely, $\Re( W)$ is only determined up to constants. A related issue is in the definition of the inverse Tilbert transform, as the Tilbert transform does not see constants. These ambiguities are built into the function spaces of \cite{BIT16}, and play a much less significant role in our analysis than in the dynamic problem. Of course, a related, but easily resolved, ambiguity is that $Q$ (and $\phi$) are only defined up to addition of a real constant.
\end{rem}

\begin{rem}\label{Properties of the conformal map.}
There are a few additional properties of $z$ that we will note, all of which have been essentially verified in the proof of \cite[Theorem 3]{BIT16}. The first is that the parameterization essentially moves ``from left to right" or, more specifically, the parameterization on top satisfies $\frac{d\alpha}{ds}>0$. This was implicitly used above in the derivation of the capillary term. Next, since $z$ is holomorphic and a diffeomorphism,  $|z_\alpha|>0$ on $S$, which combined with the asymptotics at infinity implies that there is a $\delta>0$ such that $|1+W_\alpha|=|Z_\alpha|\geq \delta$ on top. Note that we only require positivity conditions on $|1+W_\alpha|;$ the boundary being a graph would assume positivity of $1+\Re(W_\alpha)$.
\end{rem}

\subsection{The solitary wave equations}
In search for solitary wave solutions we fix a speed $c$ and make the ansatz $(Q(\alpha,t),W(\alpha,t))=(Q(\alpha-ct),W(\alpha-ct))$. The first equation in \eqref{holo} then becomes
\begin{equation}
    -cW_\alpha+F(1+W_\alpha)=0
\end{equation}
while the second equation becomes
\begin{equation}\label{second solitary}
    -cQ_\alpha+FQ_\alpha-g\mathcal{T}_h[W]+\textbf{P}_h\left[\frac{|Q_\alpha|^2}{J}\right]+\sigma \textbf{P}_h\left[i\left(\frac{W_{\alpha\alpha}}{J^{1/2}(1+W_
 \alpha)}-\frac{\overline{W_{\alpha\alpha}}}{J^{1/2}(1+\overline{W_\alpha})}\right)\right]=0.
\end{equation}

We rewrite the first equation as
\begin{equation}
    F=\textbf{P}_h\left[\frac{Q_\alpha-\overline{Q_\alpha}}{J}\right]=\frac{cW_\alpha}{1+W_\alpha}.
\end{equation}
This gives that
\begin{equation}
    \Im\left[\textbf{P}_h\left[\frac{Q_\alpha-\overline{Q_\alpha}}{J}\right]\right]=c\Im\left(\frac{W_\alpha}{1+W_\alpha}\right)=\frac{c}{J}\Im\left(W_\alpha(1+\overline{W_\alpha})\right)=\frac{c}{J}\frac{W_\alpha-\overline{W_\alpha}}{2i}.
\end{equation}
Recalling \eqref{proj form} and that the Tilbert transform maps real-valued functions to real-valued functions, we have

\begin{equation}
    \Im(\textbf{P}_hu)=\frac{1}{2}\left[\Im(u)-\mathcal{T}_h\Re(u)\right].
\end{equation}
Therefore,
\begin{equation}
     \Im\left[\textbf{P}_h\left[\frac{Q_\alpha-\overline{Q_\alpha}}{J}\right]\right]=\frac{1}{2}\Im\left(\frac{Q_\alpha-\overline{Q_\alpha}}{J}\right)=\frac{Q_\alpha-\overline{Q_\alpha}}{2iJ}.
\end{equation}
The equation we end up with is, then,
\begin{equation}
    \frac{Q_\alpha-\overline{Q_\alpha}}{2J}=\frac{c}{2}\frac{(W_\alpha-\overline{W_\alpha})}{J},
\end{equation}
which simplifies to 
\begin{equation}
   \Im(Q_\alpha)=c\Im(W_\alpha),
\end{equation}
so that
\begin{equation}\label{First solitary simple}
    Q_\alpha=cW_\alpha
\end{equation}
because $Q$ and $W$ are holomorphic.  Note that, formally, this argument only tells us that $Q_\alpha=cW_\alpha$ up to addition of a real constant. However, the decay properties of $(W_{\alpha},Q_{\alpha})$ at infinity require the constant to vanish.
\\

We now begin to simplify the second water wave equation. Beginning with \eqref{second solitary}, substituting \eqref{First solitary simple} and the definition of $F$ gives:
\begin{equation}\label{WW2 first step}
    -c^2W_\alpha+\frac{c^2W_\alpha^2}{1+W_\alpha}-g\mathcal{T}_h[W]+c^2\textbf{P}_h\left[\frac{|W_\alpha|^2}{J}\right]+\sigma \textbf{P}_h\left[i\left(\frac{W_{\alpha\alpha}}{J^{1/2}(1+W_
 \alpha)}-\frac{\overline{W_{\alpha\alpha}}}{J^{1/2}(1+\overline{W_\alpha})}\right)\right]=0.
\end{equation}
Before continuing, we note a few things. First, we have
\begin{equation}
    \textbf{P}_h\left[\frac{|W_\alpha|^2}{J}\right]=\frac{1}{2}\left[(1-i\mathcal{T}_h)\frac{|W_\alpha|^2}{J}\right].
\end{equation}
This implies that 
\begin{equation}
    \Re\left( \textbf{P}_h\left[\frac{|W_\alpha|^2}{J}\right]\right)=\frac{1}{2}\frac{|W_\alpha|^2}{J}.
\end{equation}
Therefore, taking real part of \eqref{WW2 first step} and then using the fact that holomorphic functions satisfy $\mathcal{T}_h\left[\Re(u)\right]=-\Im(u)$ we obtain:

\begin{equation}
      -c^2\Re(W_\alpha)+c^2\Re\left(\frac{W_\alpha^2}{1+W_\alpha}\right)+g\Im(W)+\frac{c^2}{2}\frac{|W_\alpha|^2}{J}+\frac{\sigma}{2}i\left( \frac{W_{\alpha\alpha}}{J^{1/2}(1+W_
 \alpha)}-\frac{\overline{W_{\alpha\alpha}}}{J^{1/2}(1+\overline{W_\alpha})}\right)=0,
\end{equation}
which can be re-written as
\begin{equation}
      -c^2\Re(W_\alpha)+c^2\Re\left(\frac{W_\alpha^2}{1+W_\alpha}\right)+g\Im(W)+\frac{c^2}{2}\frac{|W_\alpha|^2}{J}+\frac{i\sigma}{1+W_\alpha}\partial_\alpha\left(\frac{1+W_\alpha}{|1+W_\alpha|}\right)=0.
\end{equation}
After straightforward manipulation of the terms with $c^2$ we arrive at 
\begin{equation}\label{Same eq}
    -\frac{c^2}{2}\frac{(W_\alpha+\overline{W_\alpha}+W_\alpha\overline{W_\alpha})}{|1+W_\alpha|^2}+ g\Im(W)+\frac{i\sigma}{1+W_\alpha}\partial_\alpha\left(\frac{1+W_\alpha}{|1+W_\alpha|}\right)=0.
\end{equation}

As it turns out, these are exactly the same equations as the infinite-depth case considered in \cite{IT19}. However, the function spaces are different, which plays a key role. In particular, as mentioned in the introduction, there are no infinite depth pure gravity solitary waves, but there are finite depth pure gravity solitary waves.
\\

As a consistency check, we leave it as an exercise to show that \eqref{Laplace}-\eqref{Boundary phi} imply \eqref{Same eq}. 
\subsection{Notation for function spaces}
The function spaces we use are standard, and similar to \cite{IT18}. However, to set notation, we recall a few facts:
\\

Consider a standard dyadic Littlewood-Paley decomposition
$$1=\sum_{k\in \mathbb{Z}}P_k,$$
where the projectors $P_k$ select functions with frequencies $\approx 2^k$. We will place our (hypothetical) solutions in the critical Besov space $B^\frac{1}{2}_{2,1}$ defined via

$$\|u\|_{B^\frac{1}{2}_{2,1}}:=\sum_{k\geq 1}2^\frac{k}{2}\|P_k u\|_{L^2}+\|P_{\leq 0}u\|_{L^2}.$$
Our proof also makes use of the space $B^\frac{3}{2}_{2,1}$, which has the same norm as $B^\frac{1}{2}_{2,1}$ but with $2^\frac{k}{2}$ replaced by $2^\frac{3k}{2}$. Finally, we note the embedding of $B^\frac{1}{2}_{2,1}$ into $L^\infty$, and the following Moser estimate:

\begin{lem}\label{Moser}
Let $u\in B^\frac{1}{2}_{2,1}$, and suppose $G$ is a smooth function with $G(0)=0.$ Then we have the Moser estimate
\begin{equation}
    \|G(u)\|_{B^\frac{1}{2}_{2,1}}\lesssim C(\|u\|_{L^\infty})\|u\|_{B^\frac{1}{2}_{2,1}}.
\end{equation}
\end{lem}
\begin{proof}
This is a standard result. For example, it follows from \cite[Lemma 2.2]{IT18} together with the analagous Moser estimate on the level of $L^2$.
\end{proof}

\section{No solitary waves when only surface tension is present}

We are now able to state our main theorem. The result is stated in the low regularity function space $B^\frac{1}{2}_{2,1}$ defined above. However, part of the proof involves upgrading potential solutions to sufficient regularity to justify basic computations. Comparing with the infinite depth results in \cite{IT19}, our function space requires more regularity for $W_\alpha$ at low frequency, but this is to be expected, as the same happens in the dynamic problem \cite{BIT16}. From a technical standpoint, the issue is that $\mathcal{T}_h^{-1}$ does not have good mapping properties (it is not even bounded on $L^2$) compared to  the Hilbert transform, which satisfies $H^{-1}=-H.$  For justification of the other assumption - and conclusion - of \Cref{Main}, recall \Cref{Ambig} and \Cref{Properties of the conformal map.}.

\begin{thm}\label{Main}
Suppose $W_\alpha\in B^\frac{1}{2}_{2,1}$ is holomorphic, solves  \eqref{Same eq} with $g=0$ and $\sigma>0$, $|1+W_\alpha|>\delta>0$ on the top, and its extension does not vanish on $\overline{S}$. Then $W_\alpha=0$.

\end{thm}

\begin{proof}
We work with the equation
\begin{equation}\label{Same scaled}
    i\sigma\partial_\alpha\left(\frac{1+W_\alpha}{|1+W_\alpha|}\right) =c^2\left[W_\alpha+\frac{\overline{W_\alpha}}{1+\overline{W_\alpha}}\right],
\end{equation}
which holds on the top and  is just a rescaling of \eqref{Same eq} with $g=0.$ 
\\

For what follows we slightly abuse notation by not distinguishing, notationally, between $1+W_\alpha$ and its extension to the strip.
First note that since $1+W_\alpha$ is non-vanishing on the simply connected domain $S$, it admits a holomorphic logarithm. However, one has to be a little careful, to ensure that it is real on the bottom boundary. To see this, note that since, on the bottom, $1+W_\alpha$ is real, non-vanishing and has limit $1$ at infinity, it is  positive on the bottom.
\\

Define
\begin{equation}\label{log}
T:=\log(1+W_{\alpha}):=U+iV.
\end{equation}
It is easy to see that $T$ can be chosen to be holomorphic; in particular, it can be chosen to be real on the bottom. 
\\

Plugging into \eqref{Same scaled} we see that
\begin{equation}\label{Moser upgrade}
    -\sigma V_\alpha e^{iV}=c^2\left[W_\alpha+\frac{\overline{W_\alpha}}{1+\overline{W_\alpha}}\right]=c^2\left(e^{U+iV}-e^{U-iV}\right).
\end{equation}
This implies that 
\begin{equation}\label{Sinh}
    -\sigma V_\alpha=2c^2\sinh(U).
\end{equation}
Now, we upgrade regularity. By \eqref{log}, $|1+W_\alpha|>\delta,$ and \Cref{Moser}, it follows that $U,V \in B^{\frac{1}{2}}_{2,1}$. Again by Moser, we obtain 
$\sinh(U)\in  B^\frac{1}{2}_{2,1}$ which in turn implies that $V_\alpha \in B^\frac{1}{2}_{2,1}\subseteq L^2$.  From this we get $P_{>0} U_\alpha =-P_{>0}\mathcal{T}_h^{-1}V_\alpha \in B^{\frac{1}{2}}_{2,1}$. But since $U\in L^2$, it follows that $U_\alpha \in B^{\frac{1}{2}}_{2,1}$. This will be enough regularity to justify the calculations below, though $H^\infty$ regularity for $U$ and $V_\alpha$ could be obtained by reiteration.
\\

Rescaling again and using that $-V_\alpha=\mathcal{T}_hU_\alpha$, it suffices to show that the equation
\begin{equation}\label{No solutions}
    \mathcal{T}_hU_\alpha=2c^2\sinh U
\end{equation}
admits no non-zero $B_{2,1}^{\frac{3}{2}}$ solutions. For this, we let $\chi$ be a smooth function with $\chi=0$ on $(-\infty,-1]$ and $\chi=1$ on $[1,\infty)$ with $\chi'\sim 1$ on $(-\frac{1}{2},\frac{1}{2})$.  Define $\chi_r(\alpha)=\chi(\frac{\alpha}{r})$. 
\\
\\
Next, we multiply \eqref{No solutions} by $-\chi_rU_\alpha$, and obtain
\begin{equation}
    -\chi_rU_\alpha \mathcal{T}_hU_\alpha= -2c^2\chi_rU_\alpha\sinh U=-2c^2\chi_r\partial_\alpha\left( \cosh(U)-1\right).
\end{equation}
An integration by parts yields the following identity:
\begin{equation}\label{Left massaged}
    -\int_\mathbb{R}\chi_rU_\alpha \mathcal{T}_hU_\alpha d\alpha=\frac{2c^2}{r}\int_\mathbb{R}\chi'(\frac{\alpha}{r})(\cosh(U)-1)d\alpha.
\end{equation}
Now, we treat the term on the left hand side of \eqref{Left massaged}. From the product rule for the Tilbert transform we have
\begin{equation}
    \chi_r\mathcal{T}_hU_\alpha=\mathcal{T}_h(\chi_rU_\alpha)-\mathcal{T}_h(\mathcal{T}_h\chi_r\mathcal{T}_hU_\alpha)-U_\alpha \mathcal{T}_h\chi_r.
\end{equation}

Hence, using that the Tilbert transform is skew-adjoint and maps real-valued functions to real-valued functions,
\begin{equation}
\begin{split}
-\int_\mathbb{R}\chi_rU_\alpha \mathcal{T}_hU_\alpha d\alpha&=\int_\mathbb{R} U_\alpha \mathcal{T}_h(\mathcal{T}_h\chi_r \mathcal{T}_hU_\alpha )d\alpha+\int_\mathbb{R}|U_\alpha|^2\mathcal{T}_h\chi_r d\alpha-\int_\mathbb{R}U_\alpha \mathcal{T}_h(\chi_r U_\alpha)d\alpha
\\
&=\int_\mathbb{R} U_\alpha \mathcal{T}_h(\mathcal{T}_h\chi_r \mathcal{T}_hU_\alpha )d\alpha+\int_\mathbb{R}|U_\alpha|^2\mathcal{T}_h\chi_r d\alpha+\int_\mathbb{R}\chi_rU_\alpha \mathcal{T}_hU_\alpha d\alpha
\\
&=-\int_\mathbb{R} |\mathcal{T}_h U_\alpha|^2 \mathcal{T}_h\chi_r d\alpha+\int_\mathbb{R}|U_\alpha|^2\mathcal{T}_h\chi_r d\alpha+\int_\mathbb{R}\chi_rU_\alpha \mathcal{T}_hU_\alpha d\alpha.
\end{split}
\end{equation}
Hence, we obtain
\begin{equation}
    -\int_\mathbb{R}\chi_rU_\alpha \mathcal{T}_hU_\alpha d\alpha=\frac{1}{2}\int_\mathbb{R}(|U_\alpha|^2-|\mathcal{T}_hU_\alpha|^2)\mathcal{T}_h\chi_r d\alpha.
\end{equation}
Combining this with \eqref{Left massaged}, we get
\begin{equation}
    \frac{2c^2}{r}\int_\mathbb{R}\chi'(\frac{\alpha}{r})(\cosh(U)-1)d\alpha=\frac{1}{2}\int_\mathbb{R}(|U_\alpha|^2-|\mathcal{T}_hU_\alpha|^2)\mathcal{T}_h\chi_r d\alpha.
\end{equation}
The idea now is to use the fact that at low frequency, the Tilbert transform agrees with the multiplier $\xi \mapsto-hi\xi$ to third order. With this in mind, we rewrite the above equation as follows:
\begin{equation}
\begin{split}
    \frac{2c^2}{r}\int_\mathbb{R}\chi'(\frac{\alpha}{r})(\cosh(U)-1)d\alpha&=\frac{1}{2}\int_\mathbb{R}(|U_\alpha|^2-|\mathcal{T}_hU_\alpha|^2)(\mathcal{T}_h+h\partial_\alpha)\chi_r d\alpha
    \\
    &-\frac{h}{2r}\int_\mathbb{R}(|U_\alpha|^2-|\mathcal{T}_hU_\alpha|^2)\chi'(\frac{\alpha}{r})d\alpha.
\end{split}
\end{equation}
Equivalently, we have
\begin{equation}\label{low freq}
\begin{split}
    2c^2\int_\mathbb{R}\chi'(\frac{\alpha}{r})(\cosh(U)-1)d\alpha&+\frac{h}{2}\int_\mathbb{R}(|U_\alpha|^2-|\mathcal{T}_hU_\alpha|^2)\chi'(\frac{\alpha}{r})d\alpha
    \\
    &=\frac{r}{2}\int_\mathbb{R}(|U_\alpha|^2-|\mathcal{T}_hU_\alpha|^2)(\mathcal{T}_h+h\partial_\alpha)\chi_r d\alpha.
\end{split}
\end{equation}
We are now in a position to estimate the right hand side of (\ref{low freq}). Indeed, by Cauchy Schwarz and Sobolev embedding, we have,
\begin{equation}
\begin{split}
\frac{r}{2}\left|\int_\mathbb{R}(|U_\alpha|^2-|\mathcal{T}_hU_
\alpha|^2(\mathcal{T}_h+h\partial_\alpha)\chi_rd\alpha\right|&\leq Cr(\|U_\alpha\|_4^2+\|\mathcal{T}_hU_\alpha\|_4^2)\|(\mathcal{T}_h+h\partial_\alpha)\chi_r\|_2
\\
&\leq Cr\|U\|^2_{B^{\frac{3}{2}}_{2,1}}\|(\mathcal{T}_h+h\partial_\alpha)\chi_r\|_2.
\end{split}
\end{equation}
Using Plancherel's Theorem we then obtain the simple estimate,
\begin{equation}
\begin{split}
r\|U\|^2_{B^{\frac{3}{2}}_{2,1}}\|(\mathcal{T}_h+h\partial_\alpha)\chi_r\|_2 &=Cr\|U\|_{B^{\frac{3}{2}}_{2,1}}^2\|(\tanh(h\xi)-h\xi)\widehat{\chi_r}\|_2
\\
&\leq \frac{C}{r}\|U\|_{B^{\frac{3}{2}}_{2,1}}^2\bigg\|\frac{\tanh(h\xi)-h\xi}{\xi^2}\bigg\|_{L^\infty}\|\chi''(\frac{\alpha}{r})\|_2
\\
&\leq \frac{C}{r^{1/2}}\|U\|_{B^{\frac{3}{2}}_{2,1}}^2\|\chi''\|_2.
\end{split}
\end{equation}
Hence, we obtain
\begin{equation}
2c^2\int_\mathbb{R}\chi'(\frac{\alpha}{r})(\cosh(U)-1)d\alpha+\frac{h}{2}\int_\mathbb{R}(|U_\alpha|^2-|\mathcal{T}_hU_\alpha|^2)\chi'(\frac{\alpha}{r})d\alpha=\mathcal{O}_{\|U\|_{B^{\frac{3}{2}}_{2,1}}}(r^{-1/2}).
\end{equation}
Letting $r\to\infty$, dominated convergence gives
$$2c^2\int_\mathbb{R}(\cosh(U)-1)d\alpha=-\frac{h}{2}\int_\mathbb{R}(|U_\alpha|^2-|\mathcal{T}_hU_\alpha|^2)d\alpha=-\frac{h}{2}\int_\mathbb{R}|\xi|^2|\widehat{U}|^2 \text{sech}^2(h\xi)\leq 0.$$
Therefore, since $\cosh(U)-1\geq 0,$ we have 
$$\cosh(U)=1,$$
so that $U\equiv 0$. Note that taking the limit is justified because $\cosh(U)-1$ is integrable. This is thanks to the fact that $U$ is bounded, vanishes at infinity, and belongs to $L^2$.
\end{proof}

\bibliographystyle{plain}
\bibliography{cap.bib}

\begin{thebibliography}{10}

\bibitem{MR3145074}
Benjamin~F. Akers, David~M. Ambrose, and J.~Douglas Wright.
\newblock Gravity perturbed {C}rapper waves.
\newblock {\em Proc. R. Soc. Lond. Ser. A Math. Phys. Eng. Sci.},
  470(2161):20130526, 14, 2014.

\bibitem{AIT}
T.~{Alazard}, M.~{Ifrim}, and D.~{Tataru}.
\newblock A {M}orawetz inequality for water waves.
\newblock {\em ArXiv e-prints}, June 2018.
\newblock to appear in AJM.

\bibitem{AmickFraenkelToland82}
C.~J. Amick, L.~E. Fraenkel, and J.~F. Toland.
\newblock On the {S}tokes conjecture for the wave of extreme form.
\newblock {\em Acta Math.}, 148:193--214, 1982.

\bibitem{AmickToland81a}
C.~J. Amick and J.~F. Toland.
\newblock On periodic water-waves and their convergence to solitary waves in
  the long-wave limit.
\newblock {\em Philos. Trans. Roy. Soc. London Ser. A}, 303(1481):633--669,
  1981.

\bibitem{AmickToland81b}
C.~J. Amick and J.~F. Toland.
\newblock On solitary water-waves of finite amplitude.
\newblock {\em Arch. Rational Mech. Anal.}, 76(1):9--95, 1981.

\bibitem{MR963906}
Charles~J. Amick and Klaus Kirchg\"{a}ssner.
\newblock A theory of solitary water-waves in the presence of surface tension.
\newblock {\em Arch. Rational Mech. Anal.}, 105(1):1--49, 1989.

\bibitem{Beale77}
J.~Thomas Beale.
\newblock The existence of solitary water waves.
\newblock {\em Comm. Pure Appl. Math.}, 30(4):373--389, 1977.

\bibitem{BenjaminBonaBose90}
T.~B. Benjamin, J.~L. Bona, and D.~K. Bose.
\newblock Solitary-wave solutions of nonlinear problems.
\newblock {\em Philos. Trans. Roy. Soc. London Ser. A}, 331(1617):195--244,
  1990.

\bibitem{MR2073504}
B.~Buffoni.
\newblock Existence and conditional energetic stability of capillary-gravity
  solitary water waves by minimisation.
\newblock {\em Arch. Ration. Mech. Anal.}, 173(1):25--68, 2004.

\bibitem{MR2069635}
B.~Buffoni.
\newblock Existence by minimisation of solitary water waves on an ocean of
  infinite depth.
\newblock {\em Ann. Inst. H. Poincar\'e Anal. Non Lin\'eaire}, 21(4):503--516,
  2004.

\bibitem{MR1720395}
B.~Buffoni and M.~D. Groves.
\newblock A multiplicity result for solitary gravity-capillary waves in deep
  water via critical-point theory.
\newblock {\em Arch. Ration. Mech. Anal.}, 146(3):183--220, 1999.

\bibitem{MR1378603}
B.~Buffoni, M.~D. Groves, and J.~F. Toland.
\newblock A plethora of solitary gravity-capillary water waves with nearly
  critical {B}ond and {F}roude numbers.
\newblock {\em Philos. Trans. Roy. Soc. London Ser. A}, 354(1707):575--607,
  1996.

\bibitem{MR2867413}
Adrian Constantin.
\newblock {\em Nonlinear water waves with applications to wave-current
  interactions and tsunamis}, volume~81 of {\em CBMS-NSF Regional Conference
  Series in Applied Mathematics}.
\newblock Society for Industrial and Applied Mathematics (SIAM), Philadelphia,
  PA, 2011.

\bibitem{MR91075}
G.~D. Crapper.
\newblock An exact solution for progressive capillary waves of arbitrary
  amplitude.
\newblock {\em J. Fluid Mech.}, 2:532--540, 1957.

\bibitem{Crapper-pert}
P.~{de Boeck}.
\newblock {Existence of capillary-gravity waves that are perturbations of
  Crapper's waves}.
\newblock {\em ArXiv e-prints}, 2014.

\bibitem{F48}
K.~O. Friedrichs.
\newblock On the derivation of the shallow water theory. {A}ppendix to the
  formation of breakers and bores by {J}.{J}.~{S}toker.
\newblock {\em Comm. Pure Appl. Math.}, (1):81--85, 1948.

\bibitem{FriedrichsHyers54}
K.~O. Friedrichs and D.~H. Hyers.
\newblock The existence of solitary waves.
\newblock {\em Comm. Pure Appl. Math.}, 7:517--550, 1954.

\bibitem{GS72}
B.~Goawami and S.R.P Sinha.
\newblock Irrotational solitary waves with surface-tension.
\newblock {\em Proceedings of the Indian Academy of Sciences - Section A},
  76:105--112, 1972.

\bibitem{groves:survey}
M.~D. Groves.
\newblock Steady water waves.
\newblock {\em J. Nonlinear Math. Phys.}, 11(4):435--460, 2004.

\bibitem{MR1949969}
M.~D. Groves, M.~Haragus, and S.~M. Sun.
\newblock A dimension-breaking phenomenon in the theory of steady
  gravity-capillary water waves.
\newblock {\em R. Soc. Lond. Philos. Trans. Ser. A Math. Phys. Eng. Sci.},
  360(1799):2189--2243, 2002.
\newblock Recent developments in the mathematical theory of water waves
  (Oberwolfach, 2001).

\bibitem{MR2379653}
M.~D. Groves and S.-M. Sun.
\newblock Fully localised solitary-wave solutions of the three-dimensional
  gravity-capillary water-wave problem.
\newblock {\em Arch. Ration. Mech. Anal.}, 188(1):1--91, 2008.

\bibitem{MR2847283}
M.~D. Groves and E.~Wahl\'en.
\newblock On the existence and conditional energetic stability of solitary
  gravity-capillary surface waves on deep water.
\newblock {\em J. Math. Fluid Mech.}, 13(4):593--627, 2011.

\bibitem{BIT16}
Benjamin Harrop-Griffiths, Mihaela Ifrim, and Daniel Tataru.
\newblock Finite depth gravity water waves in holomorphic coordinates.
\newblock {\em Ann. PDE}, 3(1):Paper No. 4, 102, 2017.

\bibitem{Hur}
Vera~Mikyoung Hur.
\newblock No solitary waves exist on 2{D} deep water.
\newblock {\em Nonlinearity}, 25(12):3301--3312, 2012.

\bibitem{IT18}
Mihaela Ifrim and Daniel Tataru.
\newblock Two-dimensional gravity water waves with constant vorticity {I}:
  {C}ubic lifespan.
\newblock {\em Anal. PDE}, 12(4):903--967, 2019.

\bibitem{IT19}
Mihaela Ifrim and Daniel Tataru.
\newblock No solitary waves in 2{D} gravity and capillary waves in deep water.
\newblock {\em Nonlinearity}, 33(10):5457--5476, 2020.

\bibitem{MR1423002}
G\'erard Iooss and Pius Kirrmann.
\newblock Capillary gravity waves on the free surface of an inviscid fluid of
  infinite depth. {E}xistence of solitary waves.
\newblock {\em Arch. Rational Mech. Anal.}, 136(1):1--19, 1996.

\bibitem{K48}
Joseph~B. Keller.
\newblock The solitary wave and periodic waves in shallow water.
\newblock {\em Communications on Appl. Math.}, 1:323--339, 1948.

\bibitem{MR475234}
William Kinnersley.
\newblock Exact large amplitude capillary waves on sheets of fluid.
\newblock {\em J. Fluid Mech.}, 77(2):229--241, 1976.

\bibitem{KLW20}
V.~Kozlov, E.~Lokharu, and M.H. Wheeler.
\newblock Nonexistence of subcritical solitary waves.
\newblock 2020.

\bibitem{Lavrentiev43}
M.~A. Lavrentiev.
\newblock On the theory of long waves (1943); {A} contribution to the theory of
  long waves (1947).
\newblock {\em Amer. Math. Soc. Transl.}, 102:3--50, 1954.

\bibitem{MR988299}
M.~S. Longuet-Higgins.
\newblock Limiting forms for capillary-gravity waves.
\newblock {\em J. Fluid Mech.}, 194:351--375, 1988.

\bibitem{MR990170}
Michael~S. Longuet-Higgins.
\newblock Capillary-gravity waves of solitary type on deep water.
\newblock {\em J. Fluid Mech.}, 200:451--470, 1989.

\bibitem{MR2999400}
Calin~Iulian Martin.
\newblock Regularity of steady periodic capillary water waves with constant
  vorticity.
\newblock {\em J. Nonlinear Math. Phys.}, 19(suppl. 1):1240006, 7, 2012.

\bibitem{MR3020910}
Calin~Iulian Martin.
\newblock Local bifurcation for steady periodic capillary water waves with
  constant vorticity.
\newblock {\em J. Math. Fluid Mech.}, 15(1):155--170, 2013.

\bibitem{Mielke88}
Alexander Mielke.
\newblock Reduction of quasilinear elliptic equations in cylindrical domains
  with applications.
\newblock {\em Math. Methods Appl. Sci.}, 10(1):51--66, 1988.

\bibitem{MR1986694}
Dorina Mitrea and Irina Mitrea.
\newblock On the {B}esov regularity of conformal maps and layer potentials on
  nonsmooth domains.
\newblock {\em J. Funct. Anal.}, 201(2):380--429, 2003.

\bibitem{MR1869386}
Hisashi Okamoto and Mayumi Sh\={o}ji.
\newblock {\em The mathematical theory of permanent progressive water-waves},
  volume~20 of {\em Advanced Series in Nonlinear Dynamics}.
\newblock World Scientific Publishing Co., Inc., River Edge, NJ, 2001.

\bibitem{MR1133302}
P.~I. Plotnikov.
\newblock Nonuniqueness of solutions of a problem on solitary waves, and
  bifurcations of critical points of smooth functionals.
\newblock {\em Izv. Akad. Nauk SSSR Ser. Mat.}, 55(2):339--366, 1991.

\bibitem{MR2263898}
E.~I. P\u{a}r\u{a}u, J.-M. Vanden-Broeck, and M.~J. Cooker.
\newblock Nonlinear three-dimensional gravity-capillary solitary waves.
\newblock {\em J. Fluid Mech.}, 536:99--105, 2005.

\bibitem{Rus44}
J.~Scott Russell.
\newblock Report on waves.
\newblock {\em 14th meeting of the British Association for the Advancement of
  Science}, pages 311--390, 1844.

\bibitem{strauss:survey}
Walter~A. Strauss.
\newblock Steady water waves.
\newblock {\em Bull. Amer. Math. Soc. (N.S.)}, 47(4):671--694, 2010.

\bibitem{MR513927}
J.~F. Toland.
\newblock On the existence of a wave of greatest height and {S}tokes's
  conjecture.
\newblock {\em Proc. Roy. Soc. London Ser. A}, 363(1715):469--485, 1978.

\bibitem{MR2292728}
Erik Wahl\'{e}n.
\newblock Steady periodic capillary waves with vorticity.
\newblock {\em Ark. Mat.}, 44(2):367--387, 2006.

\end{thebibliography}

\end{document}